\documentclass[12pt,a4paper,twoside]{article}
\usepackage[ansinew]{inputenc}
\usepackage{amsmath}
\usepackage{amsfonts}
\usepackage{amssymb}
\usepackage{amsthm}
\usepackage{tabu}
\usepackage{tikz}
\usepackage{ifthen}
 \usetikzlibrary{shapes.geometric,arrows,snakes,backgrounds}

\makeatletter
\def\@seccntformat#1{\csname the#1\endcsname.\quad}
\makeatother

\newcommand{\OJO}{\textbf{???}\ }
\def\R{\mathbb{R}}
\def\0{\mathbf{0}}
\def\RR{\mathbb{R}^2}
\def\SS{\mathbb{S}^2}
\def\TT{\mathbb{T}^2}
\newcommand{\diam}{\mathop{\rm diam }\nolimits}
\newcommand{\Bd}{\mathop{\rm Bd}\nolimits}
\newcommand{\Inte}{\mathop{\rm Int}\nolimits}
\newcommand{\Cl}{\mathop{\rm Cl}\nolimits}
\newcommand{\Sing}{\mathop{\rm Sing}\nolimits}

\definecolor{beren}{rgb}{0.4,0.04,0.28}
\definecolor{rojo}{rgb}{0.5,0,0}
\definecolor{verde}{rgb}{0,0.4,0}
\definecolor{granate}{RGB}{130,24,130}
\definecolor{azuloscuro}{rgb}{0,0,0.8}
\definecolor{naranja}{RGB}{255,96,0}
\definecolor{marron}{rgb}{0,0.4,0}

\newcommand{\Fizq}{\LARGE $\boldsymbol{\leftleftarrows}$}
\newcommand{\Fder}{\LARGE $\boldsymbol{\rightrightarrows}$}
\newcommand{\Farr}{\LARGE $\boldsymbol{\upuparrows}$}
\newcommand{\Faba}{\LARGE $\boldsymbol{\downdownarrows}$}
\newcommand{\fizq}{\large $\boldsymbol{\leftleftarrows}$}
\newcommand{\fder}{\large $\boldsymbol{\rightrightarrows}$}
\newcommand{\farr}{\large $\boldsymbol{\upuparrows}$}
\newcommand{\faba}{\large $\boldsymbol{\downdownarrows}$}

\newcommand{\SWNE}{$\boldsymbol{\nearrow}$}
\newcommand{\NWSE}{$\boldsymbol{\searrow}$}
\newcommand{\SENW}{$\boldsymbol{\nwarrow}$}
\newcommand{\NESW}{$\boldsymbol{\swarrow}$}

\newcommand{\swne}{\tiny $\boldsymbol{\nearrow}$}
\newcommand{\nwse}{\tiny $\boldsymbol{\searrow}$}
\newcommand{\senw}{\tiny $\boldsymbol{\nwarrow}$}
\newcommand{\nesw}{\tiny $\boldsymbol{\swarrow}$}

\newcounter{letra}
\renewcommand{\theletra}{\Alph{letra}}
\newtheorem{theorem}{Theorem}[section]
\newtheorem{proposition}[theorem]{Proposition}
\newtheorem{corollary}[theorem]{Corollary}

\theoremstyle{definition}
\newtheorem{definition}[theorem]{Definition}
\theoremstyle{remark}
\newtheorem{remark}[theorem]{Remark}

\theoremstyle{plain}

\newcounter{versionfinal}
\ifthenelse{\value{versionfinal}=1}{
\newcommand{\josegines}[1]{}

\newcommand{\corregidooriginal}[2]{
}
\newcommand{\borrar}[1]{
}
\newcommand{\morado}[1]{\textcolor{black}{#1}}
\newcommand{\moradomucho}{\color{black}}
\newcommand{\negromucho}{\color{black}}
\newcommand{\apendices}[1]{
}
\renewcommand{\poneralfinal}{}
}{

\newcommand{\josegines}[1]{\renewcommand{\thefootnote}{\bfseries\color{red}\arabic{footnote}}\footnote{\textcolor{red}{\textbf{Nota de Jose Gines: } #1}}\renewcommand{\thefootnote}{\arabic{footnote}}}
\newcommand{\victor}[1]{\renewcommand{\thefootnote}{\bfseries\color{blue!50!black}\arabic{footnote}}\footnote{\textcolor{blue!50!black}{\textbf{Nota de Victor: } #1}}\renewcommand{\thefootnote}{\arabic{footnote}}}

\newcommand{\corregidooriginal}[2]{#1

{\bfseries \color{red} #2}

}

\newcommand{\borrar}[1]{
}

\newcommand{\morado}[1]{\textcolor{purple}{#1}}
\newcommand{\azul}[1]{\textcolor{blue}{#1}}

\newcommand{\moradomucho}{\color{purple}}
\newcommand{\negromucho}{\color{black}}

\newcommand{\apendices}[1]{#1}

}

\begin{document}

\title{On the Markus-Neumann theorem}
\author{Jos\'e Gin\'es Esp\'in Buend\'ia and V\'ictor Jim\'enez
Lop\'ez}
\date{\normalsize{Universidad de Murcia (Spain)}\\
\normalsize{\today}}
\maketitle

\begin{abstract}
A well-known result by L. Markus \cite{Ma}, later extended by D.
A. Neumann \cite{Ne}, states that two continuous flows on a
surface are equivalent if and only if there is a surface
homeomorphism preserving orbits and time directions of their
separatrix configurations.

In this paper we present several examples showing that, as
originally formulated, the Markus-Neumann theorem needs not work.
Besides, we point out the gap in its proof and show how to restate
it in a correct (and slightly more general) way.
\end{abstract}

\section{The Markus-Neumann theorem}

The Markus-Neumann theorem is an often cited result dealing with
the topological classification of surface flows.  Google Scholar
provides 147 explicit references to \cite{Ma}, and 91 to
\cite{Ne}, and has been used without explicit mention (mainly in
the planar case), as in \cite[p. 5]{Re}, a large number of times.
A first planar version was proved by L. Markus \cite{Ma}, under
the additional restriction of nonexistence of so-called ``limit
separatrices''. In \cite{Ne}, D. A. Neumann disposed of this
condition and extended the result to arbitrary surfaces. Roughly
speaking, the theorem states that two surface flows are equivalent
if there is a surface homeomorphism preserving a number of
distinguished orbits from both flows. However, Markus missed an
important point which, apparently, also passed unnoticed to
Neumann and the subsequent readers (see for instance \cite[pp.
33--34]{DLA}, \cite[p. 294]{Pe}, \cite[pp. 245--246]{NZ} or
\cite[pp. 225--226]{Ni}). As a consequence the theorem, as stated
in \cite{Ma,Ne}, is wrong. In fact, as we will show in the next
section, counterexamples can be found in far from pathological
settings, even for polynomial plane flows. The good news is that,
after appropriately amending the Markus-Neumann notion of
separatrix, the theorem works (and can be slightly improved). We
should stress that, when using the theorem in the polynomial
scenery, researchers typically employ an alternative, easier to
handle with, notion of separatrix (see Remark~\ref{menosmal}).
Fortunately enough, it turns out to be equivalent to our amended
definition (but not to that of Markus-Neumann'!). Therefore, all
such papers remain correct without further changes.

In what follows we list a number of notions that will be needed to
state the Markus-Neumann theorem (Theorem~\ref{mainth} below) or
in the next sections. The reader is assumed to be familiar with
the basic facts of the qualitative theory of bidimensional flows
(although some of them will be recalled below). Among others, a
good general reference on the subject is the book \cite{ABZ}.

Let $\Phi: \R \times M \to M$ be a continuous flow on a
(connected, without boundary) surface $M$. Note that $M$ is not
assumed to be compact nor orientable. We sometimes refers to is a
the couple $(M,\Phi)$.

Some specific flows will be mentioned below. Let $f_s$, $f_a$, and
$f_r$ be the planar vector fields given by $f_s(x,y)=(1,0)$,
$f_a(x,y)=(-y,x)$ and $f_r(x,y)=(-x,-y)$,  and associate to them
the corresponding planar flows $\Phi_s, \Phi_a$ and $\Phi_r$.
Also, let $f_{v_1}(x,y)= (|x|,0)$, $f_{v_2}(x,y)=(x,0)$,
$f_{v_3}(x,y)=(-x,0)$, $f_{h_1}(x,y)=(|y|,0)$ and
$f_{h_2}(x,y)=(y,0)$, being $\Phi_{v_1}$, $\Phi_{v_2}$,
$\Phi_{v_3}$, $\Phi_{h_1}$ and $\Phi_{h_2}$, respectively, their
associated planar flows.  Finally, after identifying points $(x,y)$ and
$(x',y')$ in $\R^2$ when both $x-x'$ and $y-y'$ are integers,
$f_s$ induces a flow $\Phi_{ss}$ on the torus $\TT$.

Given any $p \in M$ and any $t \in \R$, we write indistinguishably
$\Phi(t,p)=\Phi_p(t)=\Phi_t(p)$. Also, for every $p \in M$, we
call the map $\Phi_p: \R \to M$  \textit{the integral curve
associated to the point $p$}, and its image,
$\varphi_\Phi(p)=\{\Phi_p(t) : t \in \R\}$,  \textit{the orbit of
$\Phi$ through $p$}. The flow foliates $M$ as a union of points
(the singleton orbits, or \textit{singular} points of the flow),
circles (that is, homeomorphic sets to the unit circle
$\{(x,y):x^2+y^2=1\}$, corresponding to non-constant periodic
integral curves) and injective copies of the real line. By
$\Sing(\Phi)$ we denote the union set of all singular points;
nonsingular points, and the orbits containing them, are called
\textit{regular}. If $I\subset \R$ is an interval and $p \in
\Omega$, we call $\Phi_{p}(I)$ a \emph{semi-orbit} of
$\varphi_\Phi(p)$. In the particular cases $I=[t_1,t_2]$ (with
$\Phi_p(t_1)=p_1$, $\Phi_p(t_2)=p_2$), $I=[0,\infty)$ or
$I=(-\infty,0]$, we rewrite $\Phi_p(I)$ as
$\varphi_\Phi(p_1,p_2)$, $\varphi_\Phi(p,+)$ or
$\varphi_\Phi(-,p)$, respectively.

We define the \emph{$\alpha$-limit set} of the orbit
$\varphi_\Phi(p)$ (or the point $p$) as the set
 $$\alpha_\Phi(p)=\{u\in M: \exists t_n \to -\infty;
  \Phi_p(t_n)\to u\};
 $$
the \emph{$\omega$-limit set}, $\omega_\Phi(p)$, is accordingly
defined writing $t_n \to \infty$ instead. Also, we write
$\alpha_\Phi'(p)=\alpha_\Phi(p)\setminus \varphi_\Phi(p)$ and
$\omega_\Phi'(p)=\omega_\Phi(p)\setminus \varphi_\Phi(p)$. When
these sets coincide, that is, the orbit belongs neither to its
$\alpha$-limit set not its $\omega$-limit set, we call it
\emph{non-recurrent}. Usually, when there is no ambiguity on
$\Phi$, we will omit to mention $\Phi$ and write, for instance,
$\varphi(p)$ instead of $\varphi_\Phi(p)$ or $\omega(p)$ instead
of $\omega_\Phi(p)$.

We say that an orbit $\varphi(p)$  is \emph{stable} if for any
$\epsilon>0$ there is a number $\delta>0$ such that
$d(p,q)<\delta$ implies that all points from $\varphi(q)$ stay at
a distance less than $\epsilon$ from $\varphi(p)$, and it is
called \emph{unstable} otherwise. Contrary to the previous
notions, stability depends on the chosen metric; more precisely,
problems may arise when $M$ is noncompact. To avoid them we pass
to the one-point compactification $M_\infty=M\cup \{\infty\}$,
which is metrizable because $M$ is locally compact and metrizable,
and use an arbitrary distance in $M_\infty$.

If $(M, \Phi)$ is a flow, then we say that a subset $\Omega
\subset M$ is \textit{invariant} for the flow if it is the union
set of some orbits. In such a case $\Phi$ also defines a flow on
$\Omega$ ($\Phi$ can be restricted to a map from $\R \times
\Omega$ to $\Omega$), the so-called \textit{restriction of $\Phi$
to $\Omega$}, and we write $(\Omega, \Phi)$ to denote it.

We say that two flows $(M_1, \Phi_1)$ and $(M_2, \Phi_2)$ are
\textit{locally topologically equivalent} at the points $p_1\in
M_1$, $p_2\in M_2$, if there is a homeomorphism $h:U_1\to U_2$
between open neighbourhoods of $p_1$ and $p_2$, with $h(p_1)=p_2$,
carrying semi-orbits onto semi-orbits and preserving (time)
directions. When the homeomorphism maps the whole $M_1$ onto $M_2$
(hence carrying orbits onto orbits), then we call it a
\textit{topological equivalence} between $(M_1, \Phi_1)$ and
$(M_2, \Phi_2)$ and say that the flows $(M_1, \Phi_1)$ and $(M_2,
\Phi_2)$ are \emph{topologically equivalent}. As it is well known,
given a flow $(M,\Phi)$ there is a flow $(M_\infty,\Phi_\infty)$,
having $\infty$ as a singular point, whose restriction
$(M,\Phi_\infty)$ is topologically equivalent to $(M,\Phi)$ (in
fact, the orbits are the same).

Let $p$ be a singular point of $\Phi$. We say that it is
\emph{vertical} (respectively, \emph{horizontal}) if there is a
local equivalence between $\Phi$ and either $\Phi_{v_1}$,
$\Phi_{v_2}$ or $\Phi_{v_3}$ (respectively, $\Phi_{h_1}$ or
$\Phi_{h_2}$) at $p$ and $\0$.  A singular point which is neither
vertical, nor horizontal, is called \emph{essential}. Among the
essential singular points we distinguish the subset of
\textit{trivial} ones as those points which admit a neighbourhood
of singular points.

Let $p\in M$. If there is a local topological equivalence between
$\Phi$ and $\Phi_s$ at $p$ and $\0$, with corresponding
homeomorphism $h:U_1\to U_2=(-\epsilon,\epsilon)\times
(-\epsilon,\epsilon)$, then we say that $U_1$ is a \textit{tubular
neighbourhood} of $p$, and call the regions (by a \emph{region} we
mean an open connected set) into which
$h^{-1}((-\epsilon,\epsilon)\times \{0\})$ decomposes $U_1$ a
\emph{couple of lateral tubular regions} at $p$. Also, we say that
$h^{-1}(\{0\}\times (-\epsilon,\epsilon))$ is a \emph{transversal}
to $p$, which is decomposed by $p$ into a \emph{couple of lateral
transversals} at $p$. It is a well-known fact that every regular
point $p$ of $\Phi$ admits a tubular neighbourhood; on the other
hand, if $p$ is a horizontal singular point, couples of lateral
tubular regions and of lateral transversals at $p$ can be
analogously defined, just using $\Phi_{h_1}$ or $\Phi_{h_2}$
instead of $\Phi_s$. If $\mu:\R\to M$ is a continuous injective
map with the property that, for any $s\in \R$, there is
$\epsilon_s>0$ such that $\mu((s-\epsilon_s,s+\epsilon_s))$ is a
transversal to $\mu(s)$, then we call $\mu(\R)$ a
\emph{transversal} to the flow $(M,\Phi)$.

Let $\Omega$ be an invariant region for $(M,\Phi)$. Following
\cite{Ne}, we say that $\Omega$ is is \textit{parallel} when the
restriction $(\Omega, \Phi)$  is topologically equivalent to
either $(\RR, \Phi_s)$, $(\RR\setminus\{\0\}, \Phi_a)$,
$(\RR\setminus\{\0\}, \Phi_r)$ or $(\TT, \Phi_{ss})$, and use the
terms \textit{strip}, \textit{annular}, \textit{radial} and
\textit{toral}, respectively, to distinguish cases. Note that in
the toral case $\Omega=M$ is indeed a torus and all orbits are
periodic.

If $\Omega$ is parallel, and $T\subset \Omega$ is a transversal,
then we say that it is a \emph{complete} transversal to $\Omega$
provided that one the following conditions hold:
\begin{itemize}
 \item $\Omega$ is either a strip or an annular region,
 and $T$ intersects each orbit from  $\Omega$
 at exactly one point. Observe that, in the strip case,
 $T$ decomposes $\Omega$ into two regions, $\Omega_T^-$ and
 $\Omega_T^+$, corresponding to the backward and forward direction
 of the flow.
 \item $\Omega$ is a radial region and, if $p\in T$, then each of
 the two transversals into which $p$ decomposes $T$ intersects any orbit from
 $\Omega$ infinitely many times.
\end{itemize}
Also, we say that a transversal $T\subset \Omega$ is
\emph{semi-complete} when either it is complete, or 
it is one of the two transversals into which some point decomposes 
a complete 
transversal.


Let $(M, \Phi)$ be a flow and let $p \in M$. In \cite{Ne} or
\cite{Ma} separatrices are defined as follows:

\begin{definition}
 \label{bad}
 We say that an orbit $\varphi(p)$ of $(M, \Phi)$
 is \emph{ordinary} if it is neighboured by a parallel region
 $\Omega$ such that:
 \begin{itemize}
 \item[(i)] $\alpha'(q)=\alpha'(p)$ and
   $\omega'(q)=\omega'(p)$ for any $q\in \Omega$;
 \item[(ii)] $\Bd\Omega$ is the union of
   $\alpha'(p)$, $\omega'(p)$ and
   exactly two orbits $\varphi(a)$ and  $\varphi(b)$
   with $\alpha'(a)=\alpha'(b)=\alpha'(p)$ and
   $\omega'(a)=\omega'(b)=\omega'(p)$.
 \end{itemize}
 If an orbit is not ordinary, then it is called a
 \emph{separatrix}.
\end{definition}
Observe that no conditions are imposed on $\alpha'(p)$ and
$\omega'(p)$: one or both may be empty (which is to say the
infinite point $\infty$ when passing to $\Phi_\infty$), and they
may have, or not, empty intersection. On the other hand,
$\varphi(a)$ and $\varphi(b)$ must be distinct and disjoint from
$\alpha'(p)$ and $\omega'(p)$.

Let $(M,\Phi)$ be a flow and call $\mathcal{S}$ the union set of
all its separatrices. Note that the set $\mathcal{S}$ is closed.
The components of $M \setminus \mathcal{S}$ are called the
\textit{canonical regions} of $(M, \Phi)$. By a \textit{separatrix
configuration} for $(M,\Phi)$, $\mathcal{S}^+$, we mean the union
of $\mathcal{S}$ together with a representative orbit from each
canonical region.

Let $\Phi_1$ and $\Phi_2$ be two flows defined on the same surface
$M$ and let $\mathcal{S}_1$ and $\mathcal{S}_2$ be, respectively,
the union sets of their separatrices. We say that their separatrix
configurations $\mathcal{S}_1^+$ and $\mathcal{S}_2^+$ are
\textit{equivalent} if there is a homeomorphism of $M$ onto $M$,
carrying orbits of $\mathcal{S}_1^+$  onto orbits of
$\mathcal{S}_2^+$, that preserves directions.

We are ready to state the result we are concerned with in this
paper:

\begin{theorem}[the Markus-Neumann theorem \cite{Ma,Ne}]
\label{mainth}
 Let $M$ be a surface and suppose that
$\Phi_1$ and $\Phi_2$ are flows on $M$ whose sets of singular
points are discrete. Then $\Phi_1$ and $\Phi_2$ are equivalent if
and only if they have equivalent separatrix configurations.
\end{theorem}

\section{Counterexamples to the Markus-Neumann theorem \label{counter}}

As it turns out, Theorem~\ref{mainth} (as presently formulated) is
wrong, the problem being that the previous definition of
separatrix is too restrictive.  A planar counterexample is shown
by Figure~\ref{ce1}. Both flows share the orbits $\Gamma_1$,
$\Gamma_2$, $\Gamma_3$ and $\Gamma_5$ and the singular point $p$,
and the separatrices are just, in both cases, $p$, $\Gamma_1$ and
$\Gamma_2$. For instance, to show that $\Gamma_3$ is ordinary for
the right-hand flow $\Phi_2$, take an orbit $\Gamma$ enclosed by
$\Gamma_2$ but not by $\Gamma_3$, and use the strip $\Omega$
consisting of all orbits enclosed by $\Gamma_2$ but not by
$\Gamma$. Now the boundary of $\Omega$ consist, as required by
Definition~\ref{bad}, of the orbits $\Gamma_2$ and $\Gamma$, and
the singular point $p$, which is both the $\alpha$-limit and the
$\omega$-limit set of all orbits in $\Omega$ and also of
$\Gamma_2$ and $\Gamma$. (Here, as in the examples below, there is
no need to distinguish between $\alpha(q)$ and $\alpha'(q)$ nor
between $\omega(q)$ and $\omega'(q)$, because the only recurrent
orbits are the singular points). Likewise, $\Gamma_4$ is ordinary
for the left-hand flow $\Phi_1$ (use the strip $\Omega$ consisting
of all orbits enclosed by $\Gamma_2$ but not by $\Gamma_3$).

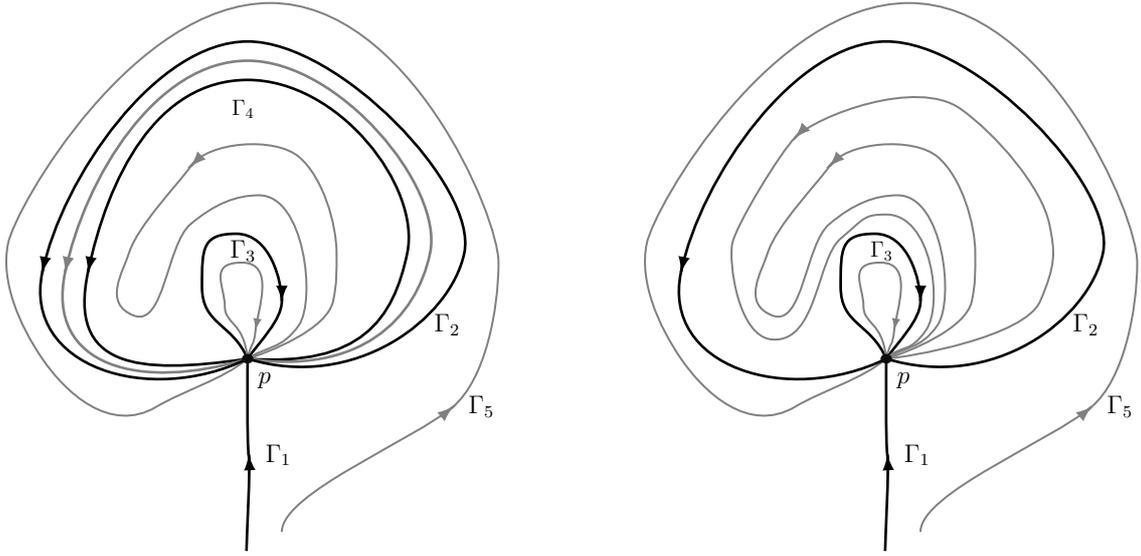
\begin{figure}
\centering
\begin{tikzpicture}[scale=0.15, yscale=0.85]

\begin{scope}[xshift=-28cm]

\draw[line width=0.75, color=gray]
   (3.0,-18.0) .. controls +(90:4) and +(225:5) ..  (18.0,-5.0)
                     .. controls +(45:4) and +(270:5) ..  (22.0, 10.0)
                     .. controls +(90:5) and +(0:12) ..  (2.0,37.0)
                     .. controls +(180:12) and +(80:4) ..  (-21.0,12.0)
                     .. controls +(260:4) and +(120:4) ..  (-18.0,0.0)
                     .. controls +(300:4) and +(215:5) ..  (-8.0,-5.0)
                     .. controls +(35:3) and +(225:3) ..  (0.0,0.0);

\draw[-latex,line width=1.0, color=gray]
    (18.0,-5.0) .. controls +(45:0.05) and +(225:0.05) ..  (18.0,-5.0);


\draw[line width=1.00]
   (-0.1,-20.0) .. controls +(90:1) and +(280:1) .. (0.1,-10.0)
           .. controls +(100:1) and +(270:0.5) ..  (0,0);
\draw[-latex, line width=1.00]
		(0.1,-10.0) .. controls +(93:0.05) and +(273:0.05) ..  (0.1,-10.0);


\draw[line width=1.00]
   (0.0,0.0) .. controls +(-5:10) and +(260:7) .. (14.0,11.0)
           .. controls +(80:7) and +(0:8) ..  (0,29.0)
                     .. controls +(180:8) and +(80:10) ..  (-14,9.0)
                    .. controls +(260:10) and +(190:10) ..  (0,0);
\draw[-latex,line width=1.00]
    (-14,9.0) .. controls +(260:0.05) and +(80:0.05) .. (-14,9.0);

\draw[line width=1.00, color=gray]
   (0.0,0.0) .. controls +(-10:10) and +(260:7) .. (16.0,11.0)
           .. controls +(80:7) and +(0:8) ..  (0,31.0)
                     .. controls +(180:8) and +(80:10) ..  (-16,9.0)
                    .. controls +(260:10) and +(200:10) ..  (0,0);
\draw[-latex,line width=1.00, color=gray]
    (-16,9.0) .. controls +(260:0.05) and +(80:0.05) .. (-16,9.0);


\draw[line width=1.00]
   (0.0,0.0) .. controls +(-20:10) and +(260:7) .. (19.0,11.0)
           .. controls +(80:7) and +(0:8) ..  (0,33.0)
                     .. controls +(180:8) and +(80:10) ..  (-18,9.0)
                    .. controls +(260:10) and +(210:10) ..  (0,0);
\draw[-latex,line width=1.00]
    (-18,9.0) .. controls +(260:0.05) and +(80:0.05) .. (-18,9.0);

\draw[line width=1.00]
   (0.0,0.0) .. controls +(60:2) and +(270:2) .. (3.0,6.0)
           .. controls +(90:2) and +(0:3) ..  (-1.0,13.0)
                     .. controls +(180:3) and +(90:4) ..  (-4,7.5)
                    .. controls +(270:4) and +(110:4) ..  (0,0);
\draw[-latex, line width=1.00]
    (3.0,6.0) .. controls +(270:0.05) and +(90:0.05) .. (3.0,6.0);

\draw[line width=0.75, color=gray]
   (0.0,0.0) .. controls +(80:2) and +(265:2) .. (0.7,3.0)
           .. controls +(85:2) and +(0:3) ..  (-0.5,10.0)
                     .. controls +(180:3) and +(90:2) ..  (-2.0,5.5)
                    .. controls +(270:2) and +(95:4) ..  (0,0);
\draw[-latex,line width=0.5, color=gray]
    (0.7,3.0) .. controls +(262:0.05) and +(82:0.05) .. (0.7,3.0);



\draw[line width=0.75, color=gray]
   (0.0,0.0) .. controls +(20:2) and +(260:3) .. (7.5,6.0)
               .. controls +(80:2) and +(-90:2) ..  (7.8,11.0)
                       .. controls +(90:3) and +(-45:3) ..  (4.5,21.0)
                       .. controls +(135:3) and +(60:3) ..  (-5.0,20.0)
             .. controls +(240:2) and +(75:3) ..  (-10.4,11.0)
                       .. controls +(255:2) and +(155:3) ..  (-10.0,4.5)
                       .. controls +(-25:3) and +(225:3) ..  (-5.0,14.0)
                       .. controls +(45:3) and +(180:2) ..(1.0,17.0)
                       .. controls +(0:4) and +(60:3) ..  (4.5,3.5)
                       .. controls +(240:3) and +(40:2) ..  (0,0);

\draw[-latex,line width=1.0, color=gray]
    (-5,20.0) .. controls +(230:0.05) and +(50:0.05) .. (-5,20.0);
		

\node[scale=0.8] at (17.5,3.5) {$\Gamma_2$};

\node[scale=0.8] at (20.5,-5.0) {$\Gamma_5$};
\node[scale=0.8] at (2.7,-10.0) {$\Gamma_1$};
\node[scale=0.8] at (-0.4,11.2) {$\Gamma_3$};
\node[scale=0.7] at (-0.4,26.0) {$\Gamma_4$};
\node[scale=0.8] at (1.5,-2.2) {$p$};

\fill[color=gray!190!white] (0,0) circle (0.5);

\end{scope}

\begin{scope}[xshift=28cm]

\draw[line width=0.75, color=gray]
   (3.0,-18.0) .. controls +(90:4) and +(225:5) ..  (18.0,-5.0)
                     .. controls +(45:4) and +(270:5) ..  (22.0, 10.0)
                     .. controls +(90:5) and +(0:12) ..  (2.0,37.0)
                     .. controls +(180:12) and +(80:4) ..  (-21.0,12.0)
                     .. controls +(260:4) and +(120:4) ..  (-18.0,0.0)
                     .. controls +(300:4) and +(215:5) ..  (-8.0,-5.0)
                     .. controls +(35:3) and +(225:3) ..  (0.0,0.0);

\draw[-latex,line width=1.0, color=gray]
    (18.0,-5.0) .. controls +(45:0.05) and +(225:0.05) ..  (18.0,-5.0);


\draw[line width=1.00]
   (-0.1,-20.0) .. controls +(90:1) and +(280:1) .. (0.1,-10.0)
           .. controls +(100:1) and +(270:0.5) ..  (0,0);
\draw[-latex, line width=1.00]
    (0.1,-10.0) .. controls +(93:0.05) and +(273:0.05) ..  (0.1,-10.0);


\draw[line width=1.00]
   (0.0,0.0) .. controls +(-20:10) and +(260:7) .. (19.0,11.0)
           .. controls +(80:7) and +(0:8) ..  (0,33.0)
                     .. controls +(180:8) and +(80:10) ..  (-18,9.0)
                    .. controls +(260:10) and +(210:10) ..  (0,0);
\draw[-latex,line width=1.00]
    (-18,9.0) .. controls +(260:0.05) and +(80:0.05) .. (-18,9.0);


%

\draw[line width=1.00]
   (0.0,0.0) .. controls +(60:2) and +(270:2) .. (3.0,6.0)
           .. controls +(90:2) and +(0:3) ..  (-1.0,13.0)
                     .. controls +(180:3) and +(90:4) ..  (-4,7.5)
                    .. controls +(270:4) and +(110:4) ..  (0,0);
\draw[-latex, line width=1.00]
    (3.0,6.0) .. controls +(270:0.05) and +(90:0.05) .. (3.0,6.0);

\draw[line width=0.75, color=gray]
   (0.0,0.0) .. controls +(80:2) and +(265:2) .. (0.7,3.0)
           .. controls +(85:2) and +(0:3) ..  (-0.5,10.0)
                     .. controls +(180:3) and +(90:2) ..  (-2.0,5.5)
                    .. controls +(270:2) and +(95:4) ..  (0,0);
\draw[-latex,line width=0.5, color=gray]
    (0.7,3.0) .. controls +(262:0.05) and +(82:0.05) .. (0.7,3.0);



\draw[line width=0.75, color=gray]
   (0.0,0.0) .. controls +(20:2) and +(260:3) .. (7.5,6.0)
               .. controls +(80:2) and +(-90:2) ..  (7.8,11.0)
                       .. controls +(90:3) and +(-45:3) ..  (4.5,21.0)
                       .. controls +(135:3) and +(60:3) ..  (-5.0,20.0)
             .. controls +(240:2) and +(75:3) ..  (-10.4,11.0)
                       .. controls +(255:2) and +(155:3) ..  (-10.0,4.5)
                       .. controls +(-25:3) and +(225:3) ..  (-5.0,14.0)
                       .. controls +(45:3) and +(180:2) ..(1.0,17.0)
                       .. controls +(0:4) and +(60:3) ..  (4.5,3.5)
                       .. controls +(240:3) and +(40:2) ..  (0,0);

\draw[-latex,line width=1.0, color=gray]
    (-5,20.0) .. controls +(230:0.05) and +(50:0.05) .. (-5,20.0);


\draw[line width=0.75, color=gray]
   (0.0,0.0) .. controls +(5:2) and +(250:4) .. (13.5,6.0)
               .. controls +(70:4) and +(-80:3) ..  (14.3,14.0)
                       .. controls +(100:4) and +(-40:4) ..  (6.5,26.0)
                       .. controls +(140:5) and +(50:3) ..  (-8.0,23.0)
             .. controls +(230:2) and +(80:2) ..  (-13.5,12.0)
                       .. controls +(260:2) and +(175:3) ..  (-10.0,2.0)
                       .. controls +(-5:4) and +(225:3) ..  (-4.0,12.5)
                       .. controls +(45:3) and +(180:2) ..(0.0,15.0)
                       .. controls +(0:4) and +(60:3) ..  (3.5,3.0)
                       .. controls +(240:3) and +(50:2) ..  (0,0);
\draw[-latex,line width=1.0, color=gray]
    (-8,23.0) .. controls +(230:0.05) and +(50:0.05) .. (-8,23.0);

\fill[color=gray!190!white] (0,0) circle (0.5);

\node[scale=0.8] at (1.5,-2.2) {$p$};
\node[scale=0.8] at (20.5,-5.0) {$\Gamma_5$};
\node[scale=0.8] at (2.7,-10.0) {$\Gamma_1$};
\node[scale=0.8] at (17.5,3.5) {$\Gamma_2$};
\node[scale=0.7] at (-0.4,11.2) {$\Gamma_3$};

\end{scope}

\end{tikzpicture}

\caption{The phase portraits of flows $\Phi_1$ (left) and
$\Phi_2$.\label{ce1}}
\end{figure}

Now, since the separatrix configurations $\mathcal{S}_1^+=
\mathcal{S}_2^+=\{p\}\cup\Gamma_1\cup\Gamma_2\cup\Gamma_3\cup\Gamma_5$
are the same, the flows $\Phi_1$ and $\Phi_2$ should be, according
to Theorem~\ref{mainth}, topologically equivalent. Clearly, they
are not: since $\Gamma_2$ is, in both cases, the maximal
homoclinic orbit, the topological equivalence should carry it onto
itself. However, there are two unstable orbits ($\Gamma_3$ and
$\Gamma_4$) inner to $\Gamma_2$ for $\Phi_1$, but just one
($\Gamma_3$) for $\Phi_2$.

\begin{remark}
As shown in \cite{EJ} (see also \cite{SS}) flows $\Phi_1$ and
$\Phi_2$ can in fact be realized by polynomial vector fields (or,
via the Bendixson projection, by analytic sphere flows with just
two singular points).
\end{remark}

An even cleaner (torus) counterexample is exhibited by
Figure~\ref{ce2}. Here, the left-hand flow $\Phi_3$ and the right
hand flow $\Phi_4$ share the orbits $\Gamma_1$ and $\Gamma_2$  and
the singular point $p$, and all orbits are homoclinic. As it
happens, $p$ is the only separatrix for both flows. To show, say,
that $\Gamma_1$ is ordinary (for $\Phi_4$), remove from $\TT$ the
closure of the strip delimited by $\Gamma_3$ and $\Gamma_4$ and
containing $\Gamma_2$, to get a radial region containing
$\Gamma_1$ with boundary $\Gamma_3\cup\Gamma_4\cup \{p\}$. In the
case of $\Phi_3$, the radial region $\Omega'=\TT\setminus
(\Gamma_2\cup \{p\})$ cannot be used (there is just one regular
orbit in its boundary), but we take off another orbit $\Gamma$ and
use the strip $\Omega=\Omega'\setminus \Gamma$ instead. Once
again, the separatrix configurations
$\mathcal{S}_3^+=\mathcal{S}_4^+=\{p\} 
\cup \Gamma_1$ coincide, but $\Phi_3$ and
$\Phi_4$ are not equivalent because $\Phi_3$ has three unstable
orbits ($p$, $\Gamma_1$ and $\Gamma_2$) and $\Phi_4$ has four
($p$, $\Gamma_1$, $\Gamma_3$ and $\Gamma_4$).

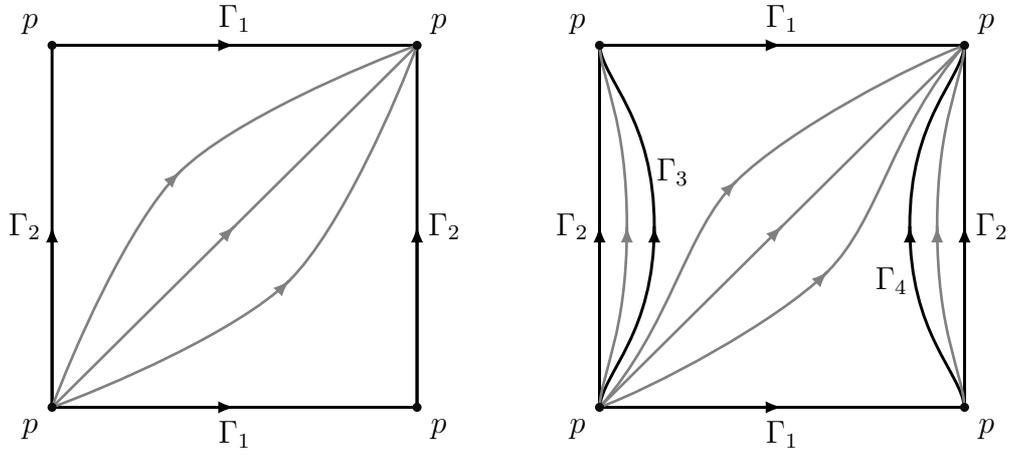
\begin{figure}
\centering
\begin{tikzpicture}[scale=.24]

\begin{scope}[xshift=0cm]

 \draw[line width=1.1] (-10.0,10.0) -- (10.0,10.0);
 \draw[line width=1.1] (-10.0,-10.0) -- (10.0,-10.0);
 \draw[line width=1.1] (-10.0,10.0) -- (-10.0,-10.0);
 \draw[line width=1.1] (10.0,10.0) -- (10.0,-10.0);

 \draw[-latex, line width=1.1] (-10.0,10.0) -- (0.0,10.0);
 \draw[-latex, line width=1.1] (-10.0,-10.0) -- (0.0,-10.0);
 \draw[-latex, line width=1.1] (-10.0,-10.0) -- (-10.0,0.0);
 \draw[-latex, line width=1.1] (10.0,-10.0) -- (10.0,0.0);

\draw[line width=1.0, color=gray]
   (-10.0,-10.0) .. controls +(45:2) and +(225:2) .. (.0,0.0)
           .. controls +(45:2) and +(225:2) ..  (10.0,10.0);
\draw[-latex, color=gray, line width=1.1]
    (0.0,0.0) .. controls +(45:0.05) and +(225:0.05) .. (0.0,0.0);

\draw[line width=1.0, color=gray]
   (-10.0,-10.0) .. controls +(70:3) and +(225:4) .. (-3.0,3.0)
           .. controls +(45:4) and +(200:3) ..  (10.0,10.0);
\draw[-latex, color=gray, line width=1.1]
    (-3.0,3.0) .. controls +(45:0.05) and +(225:0.05) .. (-3.0,3.0);

\draw[line width=1.0, color=gray]
   (-10.0,-10.0) .. controls +(20:3) and +(225:4) .. (3.0,-3.0)
           .. controls +(45:4) and +(250:3) ..  (10.0,10.0);
\draw[-latex, color=gray, line width=1.1]
    (3.0,-3.0) .. controls +(45:0.05) and +(225:0.05) .. (3.0,-3.0);

\fill[color=gray!190!white] (-10.0,10.0) circle (0.25);
\fill[color=gray!190!white] (-10.0,-10.0) circle (0.25);
\fill[color=gray!190!white] (10.0,-10.0) circle (0.25);
\fill[color=gray!190!white] (10.0,10.0) circle (0.25);

\node[scale=1.0] at (-11.5,0.0) {$\Gamma_2$};
\node[scale=1.0] at (11.5,0.0) {$\Gamma_2$};
\node[scale=1.0] at (0.0, -11.5) {$\Gamma_1$};
\node[scale=1.0] at (0.0, 11.5) {$\Gamma_1$};
\node[scale=1.0] at (-11.2,-11.2) {$p$};
\node[scale=1.0] at (-11.2,11.2) {$p$};
\node[scale=1.0] at (11.2, -11.2) {$p$};
\node[scale=1.0] at (11.2, 11.2) {$p$};
\end{scope}

\begin{scope}[xshift=30cm]
 \draw[line width=1.1] (-10.0,10.0) -- (10.0,10.0);
 \draw[line width=1.1] (-10.0,-10.0) -- (10.0,-10.0);
 \draw[line width=1.1] (-10.0,10.0) -- (-10.0,-10.0);
 \draw[line width=1.1] (10.0,10.0) -- (10.0,-10.0);

 \draw[-latex, line width=1.1] (-10.0,10.0) -- (0.0,10.0);
 \draw[-latex, line width=1.1] (-10.0,-10.0) -- (0.0,-10.0);
 \draw[-latex, line width=1.1] (-10.0,-10.0) -- (-10.0,0.0);
 \draw[-latex, line width=1.1] (10.0,-10.0) -- (10.0,0.0);

\draw[line width=1.0, color=gray]
   (-10.0,-10.0) .. controls +(45:2) and +(225:2) .. (.0,0.0)
           .. controls +(45:2) and +(225:2) ..  (10.0,10.0);
\draw[-latex, color=gray, line width=1.1]
    (0.0,0.0) .. controls +(45:0.05) and +(225:0.05) .. (0.0,0.0);

\draw[line width=1.0, color=gray]
   (-10.0,-10.0) .. controls +(49:6) and +(225:4) .. (-2.5,2.5)
           .. controls +(45:4) and +(200:3) ..  (10.0,10.0);
\draw[-latex, color=gray, line width=1.1]
    (-2.5,2.5) .. controls +(45:0.05) and +(225:0.05) .. (-2.5,2.5);

\draw[line width=1.0, color=gray]
   (-10.0,-10.0) .. controls +(20:3) and +(225:4) .. (2.5,-2.5)
           .. controls +(45:4) and +(229:6) ..  (10.0,10.0);
\draw[-latex, color=gray, line width=1.1]
    (2.5,-2.5) .. controls +(45:0.05) and +(225:0.05) .. (2.5,-2.5);

\draw[line width=1.1]
   (-10.0,-10.0) .. controls +(84:2) and +(270:6) .. (-7.0,0.0)
           .. controls +(90:6) and +(-84:2) ..  (-10.0,10.0);
\draw[-latex, line width=1.1]
    (-7.0,0.0) .. controls +(90:0.05) and +(270:0.05) .. (-7.0,0.0);
		
\draw[line width=1.0, color=gray]
   (-10.0,-10.0) .. controls +(87:2) and +(270:6) .. (-8.5,0.0)
           .. controls +(90:6) and +(-87:2) ..  (-10.0,10.0);
\draw[-latex, color=gray, line width=1.1]
    (-8.5,0.0) .. controls +(90:0.05) and +(270:0.05) .. (-8.5,0.0);

\draw[line width=1.1]
   (10.0,-10.0) .. controls +(96:2) and +(270:6) .. (7.0,0.0)
           .. controls +(90:6) and +(-96:2) ..  (10.0,10.0);
\draw[-latex, line width=1.1]
    (7.0,0.0) .. controls +(90:0.05) and +(270:0.05) .. (7.0,0.0);
		
\draw[line width=1.0, color=gray]
   (10.0,-10.0) .. controls +(93:2) and +(270:6) .. (8.5,0.0)
           .. controls +(90:6) and +(-93:2) ..  (10.0,10.0);
\draw[-latex, color=gray, line width=1.1]
    (8.5,0.0) .. controls +(90:0.05) and +(270:0.05) .. (8.5,0.0);

\fill[color=gray!190!white] (-10.0,10.0) circle (0.25);
\fill[color=gray!190!white] (-10.0,-10.0) circle (0.25);
\fill[color=gray!190!white] (10.0,-10.0) circle (0.25);
\fill[color=gray!190!white] (10.0,10.0) circle (0.25);

\node[scale=1.0] at (-11.5,0.0) {$\Gamma_2$};
\node[scale=1.0] at (11.5,0.0) {$\Gamma_2$};
\node[scale=1.0] at (0.0, -11.5) {$\Gamma_1$};
\node[scale=1.0] at (0.0, 11.5) {$\Gamma_1$};
\node[scale=1.0] at (-11.2,-11.2) {$p$};
\node[scale=1.0] at (-11.2,11.2) {$p$};
\node[scale=1.0] at (11.2, -11.2) {$p$};
\node[scale=1.0] at (11.2, 11.2) {$p$};

\node[scale=1.0] at (-6.0,3.0) {$\Gamma_3$};
\node[scale=1.0] at (6.0,-3.0) {$\Gamma_4$};

\end{scope}

\end{tikzpicture}
\caption{The phase portraits of flows $\Phi_3$ (left) and
$\Phi_4$.\label{ce2}}
\end{figure}

Clearly, the problem with the previous examples is that the
neighbouring regions we are using for ordinary orbits are, so to
speak, too ``big'', and as a consequence the bounding orbits are
not what they are ``supposed'' to be. A way to avoid this is not
allowing parallel regions to be radial (trivially they cannot be
toral either) in Definition~\ref{bad}. Moreover, we can force
strips to be ``strong''. More precisely, we say that a strip
$\Omega$ is \emph{strong} if there are non-recurrent orbits
$\Gamma_1,\Gamma_2$ such that $(\Omega', \Phi)$ is topologically
equivalent to the restriction of the flow $\Phi_s$ to $\R \times
[-1,1]$, $\Omega'=\Omega\cup\Gamma_1\cup\Gamma_2$. We call
$\Gamma_1$ and $\Gamma_2$ the \emph{border orbits} of the strip
$\Omega$, and say that a complete transversal to $\Omega$ is
\emph{strong} if it can be extended to an arc (that is, a
homeomorphic set to the interval $[0,1]$) by adding one point from
each border orbit. All orbits from a strip are non-recurrent: by
requiring that the border orbits of a strong strip also are, we
get rid of the annoying distinction between $\alpha_\Phi(p)$ and
$\alpha_\Phi'(p)$ or $\omega_\Phi(p)$ and $\omega_\Phi'(p)$ (in
the annular case, Definition~\ref{bad}(i) and (ii) are quite
redundant, anyway).

Unexpectedly, Theorem~\ref{mainth} keeps failing even after
redefining ordinary orbits  as in the paragraph above, see
Figure~\ref{ce3}. In this torus example, common orbits to $\Phi_5$
and $\Phi_6$ are $\Gamma_1$, $\Gamma_2$, $\Gamma_3$ and $\Gamma_5$
and the singular points $p$ and $q$. Observe that typical orbits
of these flows have $\Gamma_1\cup \{p\}$  as their $\alpha$-limit
set and $\{p\}$ as their $\omega$-limit set. Checking that
$\Phi_5$ and $\Phi_6$ are not equivalent, while having the same
separatrix configurations $\mathcal{S}_5^+=
\mathcal{S}_6^+=\{p,q\}\cup\Gamma_1\cup\Gamma_2\cup\Gamma_3\cup\Gamma_5$,
is left to the reader's care.

\begin{figure}
\centering
\begin{tikzpicture}[scale=.24]

\begin{scope}[xshift=0cm, yscale=2.0/3.0]
 \draw[line width=0.3, color=gray] (-10.0,15.0) -- (-10.0,-15.0);
 \draw[line width=0.3, color=gray] (10.0,15.0) -- (10.0,-15.0);
 \draw[line width=1.1] (-10.0,0.0) -- (-10.0,-15.0);
 \draw[line width=1.1] (10.0,0.0) -- (10.0,-15.0);
 \draw[line width=1.1] (-10.0,15.0) -- (10.0,15.0);
 \draw[line width=1.1] (-10.0,-15.0) -- (10.0,-15.0);


 \draw[-latex, line width=1.1] (-10.0,15.0) -- (0.0,15.0);
 \draw[-latex, line width=1.1] (-10.0,-15.0) -- (-0.0,-15.0);
 \draw[-latex, line width=1.1] (-10.0,0.0) -- (-10.0,-8.0);
 \draw[-latex, line width=1.1] (10.0,0.0) -- (10.0,-8.0);

\draw[line width=1.1] (-10.0,14.0) .. controls +(-5:2) and +(172:2) .. (0.0,13.0)
           .. controls +(-8:2) and +(170:2) ..  (10.0,11.0);
\draw[-latex, line width=1.1] (0.0,13.0) .. controls +(-8:0.05) and +(172:0.05) .. (0.0,13.0);

\draw[line width=1.1] (-10.0,11.0) .. controls +(-10:4) and +(155:3) .. (0.0,8.5) 
           .. controls +(-25:3) and +(135:3) .. (7.0,4.5)
           .. controls +(-45:2) and +(110:2) ..  (10.0,0.0);
\draw[-latex, line width=1.1] (7.0,4.5) .. controls +(-45:0.05) and +(135:0.05) .. (7.0,4.5);


\draw[line width=0.75, color=gray] (-10.0,8.5) .. controls +(-10:2) and +(90:7) .. (0.0,0.0)
           .. controls +(270:7) and +(150:2) ..  (10.0,-15);
\draw[-latex, line width=1.1, color=gray] (0.0,0.0) .. controls +(270:0.05) and +(90:0.05) .. (0.0,0.0);

\draw[line width=0.75, color=gray] (-10.0,13.2) .. controls +(-2:2) and +(170:3) .. (0.0,11.8)
           .. controls +(-10:3) and +(150:2) ..  (10.0,8.5);
\draw[-latex, line width=1.1, color=gray] (0.0,11.8) .. controls +(0:0.05) and +(180:0.05) .. (0.0,11.8);


\draw[line width=0.75, color=gray] (-10.0,10.0) .. controls +(-3:5) and +(110:6) .. (7.0,0.0)
           .. controls +(290:6) and +(120:2) ..  (10.0,-15);
\draw[-latex, line width=1.1, color=gray] (7.0,0.0) .. controls +(290:0.05) and +(110:0.05) .. (7.0,0.0);

\draw[line width=0.75, color=gray] (-10.0,6.0) .. controls +(-20:2) and +(90:5) .. (-7.0,-7.0)
           .. controls +(270:5) and +(170:2) ..  (10.0,-15);
\draw[-latex, line width=1.1, color=gray] (-7.0,-7.0) .. controls +(270:0.05) and +(90:0.05) .. (-7.0,-7.0);

\end{scope}

\begin{scope}[xshift=0cm]
\fill[color=gray!190!white] (-10.0,10.0) circle (0.25);
\fill[color=gray!190!white] (-10.0,-10.0) circle (0.25);
\fill[color=gray!190!white] (10.0,-10.0) circle (0.25);
\fill[color=gray!190!white] (10.0,10.0) circle (0.25);
\fill[color=gray!190!white] (10.0,0.0) circle (0.25);
\fill[color=gray!190!white] (-10.0,0.0) circle (0.25);

\node[scale=1.0] at (-11.5,-5.0) {$\Gamma_2$};
\node[scale=1.0] at (11.5,-5.0) {$\Gamma_2$};
\node[scale=1.0] at (-11.2,0.0) {$q$};
\node[scale=1.0] at (11.2,0.0) {$q$};
\node[scale=1.0] at (0.0, -11.5) {$\Gamma_1$};
\node[scale=1.0] at (0.0, 11.5) {$\Gamma_1$};
\node[scale=1.0] at (-11.2,-11.2) {$p$};
\node[scale=1.0] at (-11.2,11.2) {$p$};
\node[scale=1.0] at (11.2, -11.2) {$p$};
\node[scale=1.0] at (11.2, 11.2) {$p$};

\node[scale=0.8] at (-1.5,0.0) {$\Gamma_5$};
\node[scale=0.8] at (0.0,6.8) {$\Gamma_5$};
\node[scale=0.8] at (8.2,1.3) {$\Gamma_3$};
\node[scale=0.8] at (8.2,8.9) {$\Gamma_3$};

\end{scope}

\begin{scope}[xshift=30cm, yscale=2.0/3.0]

 \draw[line width=0.3, color=gray] (-10.0,15.0) -- (-10.0,-15.0);
 \draw[line width=0.3, color=gray] (10.0,15.0) -- (10.0,-15.0);
 \draw[line width=1.1] (-10.0,0.0) -- (-10.0,-15.0);
 \draw[line width=1.1] (10.0,0.0) -- (10.0,-15.0);
 \draw[line width=1.1] (-10.0,15.0) -- (10.0,15.0);
 \draw[line width=1.1] (-10.0,-15.0) -- (10.0,-15.0);


 \draw[-latex, line width=1.1] (-10.0,15.0) -- (0.0,15.0);
 \draw[-latex, line width=1.1] (-10.0,-15.0) -- (-0.0,-15.0);
 \draw[-latex, line width=1.1] (-10.0,0.0) -- (-10.0,-8.0);
 \draw[-latex, line width=1.1] (10.0,0.0) -- (10.0,-8.0);

\draw[line width=1.1] (-10.0,14.0) .. controls +(-5:2) and +(172:2) .. (0.0,13.0)
           .. controls +(-8:2) and +(170:2) ..  (10.0,11.0);
\draw[-latex, line width=1.1] (0.0,13.0) .. controls +(-8:0.05) and +(172:0.05) .. (0.0,13.0);

\draw[line width=1.1] (-10.0,11.0) .. controls +(-10:4) and +(155:3) .. (0.0,8.5) 
           .. controls +(-25:3) and +(135:3) .. (7.0,4.5)
           .. controls +(-45:2) and +(110:2) ..  (10.0,0.0);
\draw[-latex, line width=1.1] (7.0,4.5) .. controls +(-45:0.05) and +(135:0.05) .. (7.0,4.5);


\draw[line width=0.75, color=gray] (-10.0,8.5) .. controls +(-10:2) and +(90:7) .. (0.0,0.0)
           .. controls +(270:7) and +(150:2) ..  (10.0,-15);
\draw[-latex, line width=1.1, color=gray] (0.0,0.0) .. controls +(270:0.05) and +(90:0.05) .. (0.0,0.0);

\draw[line width=0.75, color=gray] (-10.0,13.2) .. controls +(-2:2) and +(170:3) .. (0.0,11.8)
           .. controls +(-10:3) and +(150:2) ..  (10.0,8.5);
\draw[-latex, line width=1.1, color=gray] (0.0,11.8) .. controls +(0:0.05) and +(180:0.05) .. (0.0,11.8);


\draw[line width=0.75, color=gray] (-10.0,10.0) .. controls +(-3:5) and +(110:6) .. (7.0,0.0)
           .. controls +(290:6) and +(120:2) ..  (10.0,-15);
\draw[-latex, line width=1.1, color=gray] (7.0,0.0) .. controls +(290:0.05) and +(110:0.05) .. (7.0,0.0);

\draw[line width=0.75, color=gray] (-10.0,7.5) .. controls +(-15:2) and +(100:5) .. (-4.0,-6.0)
           .. controls +(280:5) and +(170:2) ..  (10.0,-15);
\draw[-latex, line width=1.1, color=gray] (-4.0,-6.0) .. controls +(280:0.05) and +(100:0.05) .. (-4.0,-6.0);

%
%

		
\draw[line width=0.75, color=gray]
    (-10.0,4.0) .. controls +(-70:3) and +(90:3) .. (-8.5,-8.0)
                .. controls +(270:3) and +(70:3) ..  (-10.0,-15.0);
\draw[-latex, line width=1.1, color=gray]
    (-8.5,-8.0) .. controls +(270:0.05) and +(90:0.05) .. (-8.5,-8.0);
		
\draw[line width=0.75, color=gray] (-10.0,11.5) .. controls +(-8:5) and +(165:2) .. (0.0,9.3) 
           .. controls +(-15:2) and +(155:1) .. (7.0,5.9)
           .. controls +(-25:1) and +(150:2) ..  (10.0,4.0);
\draw[-latex, line width=1.1, color=gray] (7.0,5.9) .. controls +(-25:0.05) and +(155:0.05) .. (7.0,5.9);		

\draw[line width=1.1]
    (-10.0,5.5) .. controls +(-60:3) and +(90:3) .. (-7.0,-8.0)
                .. controls +(270:3) and +(60:3) ..  (-10.0,-15.0);
\draw[-latex, line width=1.1]
    (-7.0,-8.0) .. controls +(270:0.05) and +(90:0.05) .. (-7.0,-8.0);

\draw[line width=1.1] (-10.0,12.0) .. controls +(-3:2) and +(170:3) .. (0.0,10.6)
           .. controls +(-10:3) and +(140:2) ..  (10.0,5.5);
\draw[-latex, line width=1.1] (0.0,10.6) .. controls +(0:0.05) and +(180:0.05) .. (0.0,10.6);
		

%
%
		%
		%
%
		%
%
		%
%
\end{scope}

\begin{scope}[xshift=30cm]
\fill[color=gray!190!white] (-10.0,10.0) circle (0.25);
\fill[color=gray!190!white] (-10.0,-10.0) circle (0.25);
\fill[color=gray!190!white] (10.0,-10.0) circle (0.25);
\fill[color=gray!190!white] (10.0,10.0) circle (0.25);
\fill[color=gray!190!white] (10.0,0.0) circle (0.25);
\fill[color=gray!190!white] (-10.0,0.0) circle (0.25);

\node[scale=1.0] at (-11.5,-5.0) {$\Gamma_2$};
\node[scale=1.0] at (11.5,-5.0) {$\Gamma_2$};
\node[scale=1.0] at (-11.2,0.0) {$q$};
\node[scale=1.0] at (11.2,0.0) {$q$};
\node[scale=1.0] at (0.0, -11.5) {$\Gamma_1$};
\node[scale=1.0] at (0.0, 11.5) {$\Gamma_1$};
\node[scale=1.0] at (-11.2,-11.2) {$p$};
\node[scale=1.0] at (-11.2,11.2) {$p$};
\node[scale=1.0] at (11.2, -11.2) {$p$};
\node[scale=1.0] at (11.2, 11.2) {$p$};

\node[scale=0.8] at (-1.5,0.0) {$\Gamma_5$};
\node[scale=0.8] at (8.2,1.3) {$\Gamma_3$};
\node[scale=0.8] at (8.2,8.9) {$\Gamma_3$};
\node[scale=1.0] at (-5.5,-5.0) {$\Gamma_4$};
\node[scale=0.8] at (8.7,15.1/3.0) {$\Gamma_4$};

\end{scope}

\end{tikzpicture}
\caption{The phase portraits of flows $\Phi_5$ (left) and
$\Phi_6$.\label{ce3}}
\end{figure}
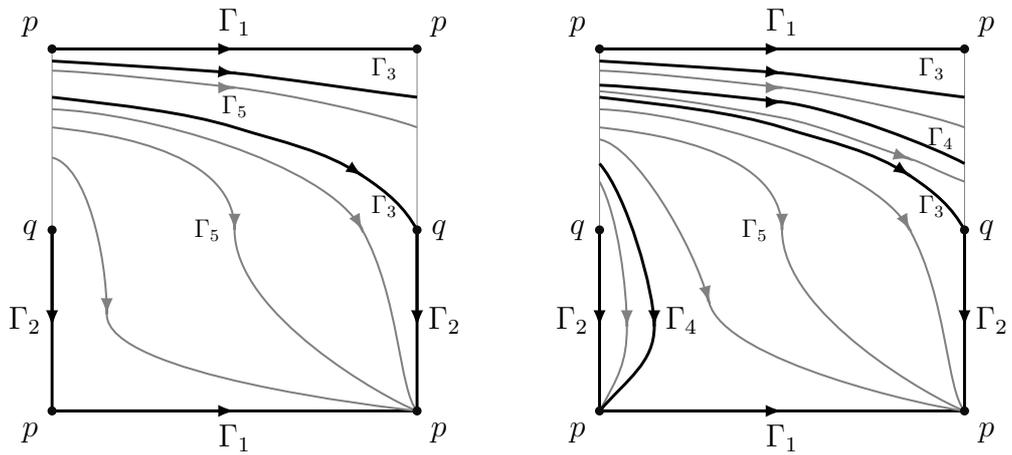

The underlying problem here is that $\alpha$-limit and
$\omega$-limit sets should be separately managed by
Definition~\ref{bad}, but they are not. The orbit $\Gamma_4$ is
neighboured by strong strips, as close to it as required, whose
boundaries consist of (as prescribed) the border orbits, the
$\alpha$-limit set $\Gamma_1\cup \{p\}$ and, \emph{a fortiori},
the $\omega$-limit set $\{p\}$. Nevertheless, after removing a
strong transversal from the strip, the boundary of the forward
semi-strip contains, besides the border semi-orbits, the strong
transversal and the $\omega$-limit point $\{p\}$, the ``spurious''
orbit $\Gamma_1$.

Taking all the above into consideration, we define:

\begin{definition}
 \label{proper}
 We say that an orbit $\varphi(p)$ of $(M, \Phi)$
 is \emph{almost fine}  if it is
 neighboured by an annular
 region or by a strong strip $\Omega$ with border orbits
 $\varphi(a)$ and
 $\varphi(b)$ such that:
 \begin{itemize}
 \item[(i)] $\alpha(q)=\alpha(p)$ and
   $\omega(q)=\omega(p)$ for any
   $q\in \Omega\cup\varphi(a)\cup\varphi(b)$;
 \item[(ii)]
   $\Bd\Omega=\varphi(a)\cup\varphi(b)\cup\alpha(p)
   \cup\omega(p)$.
 \end{itemize}
 If $\varphi(p)$ satisfies the analogous conditions, replacing
 (ii) by
 \begin{itemize}
 \item[(ii')] if $T$ is a strong transversal to $\Omega$
   with endpoints $a$ and $b$,
     then $\Bd\Omega_T^-=T\cup\varphi(-\infty,a)\cup
   \varphi(-\infty,b)\cup\alpha(p)$ and
  $\Bd\Omega_T^+=T\cup\varphi(a,\infty)\cup
   \varphi(b,\infty)\cup\omega(p)$,
 \end{itemize}
 then we say that $\varphi(p)$ is \emph{fine}.

If an orbit is not fine, then it is called a \emph{separator}.
\end{definition}

Observe that the union set of all separators is closed as well,
when the components of its complementary set will be called
\emph{standard regions}. Since all separatrices are separators,
every standard region is contained in a canonical region. The
notions of \emph{separator configuration} and of \emph{equivalence
of separator configurations} are accordingly defined.

\begin{remark}
 \label{menosmal}
Typically, books and papers invoking the Markus-Neumann theorem in
the setting of analytic sphere flows (in particular, after
carrying polynomial planar flows to the sphere via the Bendixson
or the Poincar\'e projections), use an alternative definition of
separatrix, see for instance \cite[Section~3.11]{Pe}. Here, under
the additional assumption of finiteness of singular points, an
orbit is called a ``separatrix'' if and only if it is either a
singular point, a limit cycle, or an orbit lying in the boundary
of an hyperbolic sector. Using the finite sectorial decomposition
property for isolated (non-centers) singular points of analytic
flows, noting that analyticity excludes the existence of one-sided
isolated periodic orbits, and recalling some basic
Poincar\'e-Bendixson theory, it is not difficult to show that this
notion is, in fact, equivalent to that of separator (see also
Proposition~\ref{casi}(b) or (c) below). Of course, as emphasized
by our first counterexample (and contrarily to that stated in
\cite{Pe}) there may be orbits bounding hyperbolic sectors which
are not separatrices in the Markus-Neumann sense.

Note, finally, that the previous discussion make no sense outside
the sphere (just think of the irrational flow on the torus: here
all orbits are separatrices).
\end{remark}

Trivially, a fine orbit is almost fine. The converse is not true,
as shown by the flow $\Phi_6$. Nevertheless, we have:

\begin{proposition}
 \label{casi}
 Let $\varphi(p)$ be an almost fine orbit of $(M,\Phi)$.
 Assume that one of the following conditions holds:
 \begin{itemize}
 \item[(a)] Both $\alpha(p)$ and $\omega(p)$ are finite (that is,
  empty or consisting of one point);
 \item[(b)] $M$ has zero genus and the set of essential singular
  points of $\Phi$ is totally disconnected;
 \item[(c)] $M=\RR$ or $M=\SS$.
 \end{itemize}
 Them $\varphi(p)$ is fine.
\end{proposition}

\begin{proof}
In all three cases we must show that if $\Omega$ satisfies (i) and
(ii) in Definition~\ref{proper}, then (ii') holds as well.

Assume that (a) holds. We just prove (the other equality is
analogous) $\Bd\Omega_T^-=T\cup\varphi(-\infty,a)
\cup\varphi(-\infty,b)\cup\alpha(p)$, when we can assume
(otherwise the statement is trivial) $\alpha(p)=\{u\}\neq
\omega(p)$. Then there is a small topological disk $D$ (that is, a
homeomorphic set to the unit disk $\{(x,y)\in \R^2: x^2+y^2\leq
1\}$), neighbouring $u$ (hence not intersecting $\omega(p)$), such
that $\varphi(a)$ and $\varphi(b)$ intersect $D$ at respective
semi-orbits $\varphi(-,a')$, $\varphi(-,b')$; moreover, this can
be done so that one of the arcs in $\Bd D$ joining $a'$ and $b'$
is the closure of a strong transversal $T'$ to $\Omega$. One of
the regions into which $\varphi(-,a')\cup \{u\}\cup\varphi(-,b')$
decomposes $\Inte D$ cannot intersect $\Omega$, and the other one,
call it $\Omega'$, includes $\Omega_{T'}^-$. By the hypothesis on
$\Bd \Omega$, and the fact that $\omega(p)$ does not intersect
$D$, both $\Omega'$ and $\Omega_{T'}^-$ have the same boundary,
hence (because they are connected) $\Omega'=\Omega_{T'}^-$. This
implies the statement.

Now suppose that (b) holds. Because $M$ has zero genus, there is
no loss of generality in assuming that it is a region in $\SS$.
Observe that, since $\varphi(p)$ is almost fine, all non-essential
singular points contained in $\alpha(p)\cup\omega(p)$ must be
horizontal. If $\alpha(p)\cup\omega(p)$ contains a regular point
or an horizontal singular point, then a standard
Poincar\'e-Bendixson argument allows to find a semi-orbit
$\varphi(c,d)$ of $\varphi(p)$, and a transversal joining $c$ and
$d$, whose union is a circle decomposing $\SS$ into two regions,
one including $\alpha(p)$, the other one including $\omega(p)$. By
the compactness of $\{a,b\}\cup T$, there is $t_0$ such that
$\Omega_{\Phi_{-t_0}(T)}^-$ is included in the first region, while
$\Omega_{\Phi_{t_0}(T)}^+$ is included in the second one, which
easily implies that $\varphi(p)$ is fine. In the case when all
points from $\alpha(p)$ and $\omega(p)$ are singular and
essential, total disconnectedness implies finiteness and (a)
applies.

Finally, suppose that (c) is true. If suffices to consider the
case $M=\SS$, as then the case $M=\RR$ follows by passing to its
one-point compactification, which is precisely $\SS$. Moreover, as
in (b), we can additionally assume that all points from
$\alpha(p)\cup \omega(p)$ are singular. Let $S$ denote the set of
singular points of $\Phi$, and let $U$ be the component of
$\SS\setminus S$ including $\varphi(p)$. Next define the
equivalence relation $\sim$ in $\SS$ by $u\sim v$ if $u=v$ or
there is a component $C$ of $\SS\setminus U$ such that $u,v\in C$.
As explained in \cite[p. 481]{BJ}, the quotient space $\SS_\sim$
is homeomorphic to $\SS$ and the flow $\Phi$ collapses, in the
natural way, to a flow $\Phi_\sim$ on $\SS_\sim$, whose set of
singular points is totally disconnected. By applying (b) to the
collapsed flow, we deduce that the distances
$d(\Phi_t(q),\alpha(p))$, $q\in \Omega_T^-$, tend uniformly to
zero as $t\to -\infty$, and the same is true for
$d(\Phi_t(q),\omega(p))$, $q\in \Omega_T^+$ and $t\to \infty$.
Therefore, $\varphi(p)$ is fine.
\end{proof}

After extending the notion of separatrix as described in
Definition~\ref{proper}, Theorem~\ref{mainth} works and, in fact,
can be slightly improved, see Theorem~\ref{elbueno} below. The
improvement has to do with essential singular points. The
left-hand flow $\Phi_7$ from Figure~\ref{ce4} is that associated
(after deformation to clarify the picture outside the unit circle)
to the vector field
$$f_7(x,y)=(1- x^2 - y^2) \left(-(1- x^2 -
y^2) x - y, x - (1- x^2 - y^2) y\right),
$$
having the origin and the unit circle $\mathbb{S}^1$ as its set of
singular points. Consecutive points of the semi-orbit starting at
$(2,0)$ and intersecting the positive $x$-semiaxis (respectively,
negative $x$-semiaxis, positive $y$-semiaxis) are denoted by
$(a_n)_{n=1}^\infty$ (respectively, $(b_n)_{n=1}^\infty$,
$(c_n)_{n=1}^\infty$). To construct $\Phi_8$ we modify, as
indicated in the picture, the semi-orbits from the regions
enclosed by $\varphi_{\Phi_7}(a_n,b_n),
\varphi_{\Phi_7}(a_{n+1},b_{n+1})$, and the segments connecting
$a_n$ and $a_{n+1}$ and $b_n$ and $b_{n+1}$. In the lower
half-plane, and inside $\mathbb{S}^1$, the phase portrait does not
change. Thus, for both flows, all regular orbits  spiral towards
$\mathbb{S}^1$ in positive time, and if $\Gamma_1$ denotes the
orbit passing through $(2,0)$ and $\Gamma_2$ is an orbit inside
$\mathbb{S}^1$, then
$\{\0\}\cup\mathbb{S}^1\cup\Gamma_1\cup\Gamma_2$ is a 
separator configuration for both $\Phi_7$ and $\Phi_8$.

\begin{figure}
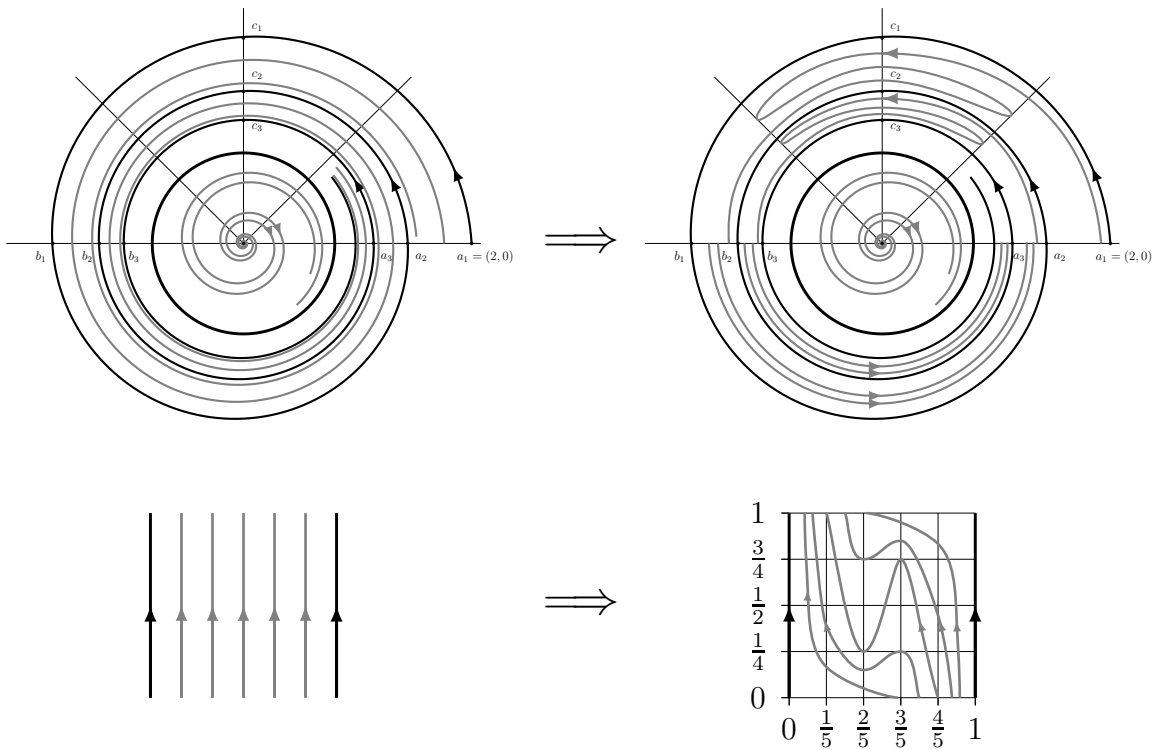

\centering

\caption{The phase portraits of flows $\Phi_7$ (left) and
$\Phi_8$.\label{ce4}}
\end{figure}

Nevertheless, these flows are not equivalent. The key point is
that, while all semi-lines starting from the origin are
transversal (except at the unit circle) to $\Phi_7$, transversals
connecting the points $(c_n)_n$ for $\Phi_8$ cannot be fully
included in the octant $\{(x,y):y>|x|\}$, hence their diameters
are uniformly bounded, from below, by a positive number. Now
assume that $h$ is a topological equivalence mapping the orbits of
$\Phi_8$ onto those of $\Phi_7$. Then the points $h(c_{n})$
converge to a point $d\in \mathbb{S}^1$, and there are
transversals $T_n$ (for $\Phi_7$) connecting the points $h(c_n)$
and $h(c_{n+1})$ whose diameters tend to zero. Hence the diameters
of the transversals $h^{-1}(T_n)$ (for $\Phi_8$) also tend to
zero, which is impossible because they connect the points  $c_n$
and $c_{n+1}$.

There is no contradiction with Theorem~\ref{mainth} here, because
the set of singular points is not discrete (by a \emph{discrete
set} we mean a set consisting of isolated points). On the other
hand, note that all singular points in the octant are essential
for the flow $\Phi_8$, which is \emph{really} the reason why
Theorem~\ref{mainth} fails in this case:

\begin{theorem}
 \label{elbueno}
 Let $M$ be a surface and suppose that
 $\Phi_1$ and $\Phi_2$ are flows on $M$ whose sets of essential
 singular points are discrete. Then $\Phi_1$ and $\Phi_2$ are equivalent if and
 only if  they have equivalent separator
 configurations.
\end{theorem}

\begin{remark} If a continuous flow on $M$ possesses a trivial
singular point, then it is clear that its set of essential
singular points cannot be discrete.
\end{remark}

We outline the proof of Theorem~\ref{elbueno} in the next section.
Let us presently emphasize the usefulness of the improved
condition on the singular points. On the one hand, recall that it
implies, in the zero genus case, that all almost fine orbits are
fine (Proposition~\ref{casi}(b)). On the other hand, we have:

\begin{proposition}
 \label{analytic}
 If $(M,\Phi)$ is associated to an
analytic vector field $f$, then either $f$ is identically zero or
its set of essential singular points is discrete.
\end{proposition}

\begin{proof} If $\Sing(\Phi)$ has non-empty interior, the
analyticity of $f$ and the connectedness of $M$ implies that
$\Sing(\Phi)=M$; we discard this trivial case in the rest of the
proof.

It is a well-known fact that, for every $p \in M$, there exist a
neighbourhood $U_p$ of $p$, an analytic map $\rho_p:U_p \to \R$
and an analytic vector field $g_p$ on $U_p$, such that the
restriction of $f$ to $U_p$ equals $\rho_p g_p$ and the vector
field $g_p$ has no zeros in $U_p \setminus \{p\}$ (see, e.g.,
\cite[Theorem~4.5]{jl07}). If $S_1$ denotes the set of singular
points with the property that any $g_p$ in such a decomposition
vanishes at $p$, then $S_1$ is clearly closed and discrete.

Since $\Sing(\Phi)$ is the set of zeros of the analytic map
$f_x^2+f_y^2$, $f=(f_x,f_y)$, the theory of Puiseux series implies
that $\Sing(\Phi)$ is locally, at each of its points $p$, a
topological star with finitely many branches (``zero'' branches
meaning that the point is isolated in $\Sing(\Phi)$); moreover,
any such branch $B$ can be parametrized via a bijective analytic
map $\varphi:[0,1]\to B$ with $\varphi(0)=p$ (for a proof see,
e.g., ~\cite[Theorem~4.3]{jl07}). We emphasize that $\varphi$ is
analytic in the whole closed interval (that is, it can be
analytically extended to a larger open interval containing
$[0,1]$). Clearly, the set $S_2$ of points where $\Sing(\Phi)$ is
not locally a $2$-star (that is, there is not an arc neighbouring
$p$ in $\Sing(\Phi)$) is also closed and discrete.

To finish the proof, it then suffices to show that if $S_3$ is the
set of essential singular points not included in $S_1\cup S_2$,
then $S_3$ is closed and discrete as well. Let $p\in S_3$ and
assume $U_p$ to be small enough so that it is a tubular
neighbourhood of $p$ for the non-vanishing vector field $g_p$. We
can also assume that there is an analytic bijection
$\lambda_p:(-1,1)\to \Sing(\Phi)\cap U_p$. Since $p$ is not
vertical, $\lambda_p'(0)$ must be parallel to $g_p(p)$, that is,
$T_p(s)=\lambda_{p,x}'(s)g_{p,y}(\lambda_p(s))
-\lambda_{p,y}'(s)g_{p,x}(\lambda_p(s))$ vanishes for $s=0$, and
since $p$ is not horizontal, $T_p(s)$ cannot be identically zero.
Analyticity then implies that there is $\epsilon>0$ such that
$T_p(s)$ does not vanishes at $(-\epsilon,\epsilon)\setminus
\{0\}$, that is, all singular points close enough to $p$ are
vertical. In particular, $S_3$ is discrete.

To prove that $S_3$ is closed, it suffices to show that if
$(p_n)_n$ is a sequence of pairwise distinct points of $S_3$, then
it cannot converge. Assume the opposite and call $p$ its limit,
when we can also assume that all points $p_n$ belong to the same
branch $B$ of the star of singular points with centre $p$ and are
included in the neigbourhood $U_p$.  Find an analytic
parametrization $\varphi:[0,1]\to B$ as previously explained, with
$\varphi(t_n)=p_n$ and $t_n\to 0$, and realize that vectors
$\varphi'(t_n)$ and $g_p(\varphi(t_n))$ are parallel for all $n$.
Hence, $\varphi'(t)$ and $g_p(\varphi(t))$ are parallel for all
$t\in [0,1]$, which is to say that all points $p_n$ are, in fact,
horizontal. This contradiction finishes the proof.
 \end{proof}

\begin{corollary}
 Let $M$ be a surface and suppose that
 $\Phi_1$ and $\Phi_2$ are flows on $M$ associated to
 analytic vector fields. Then $\Phi_1$ and $\Phi_2$ are
equivalent if and only if
 they have equivalent separator
configurations.
\end{corollary}

\section{Why the proof of Theorem~\ref{mainth} fails,
and how to
prove Theorem~\ref{elbueno}\label{amend}}

Roughly speaking, the proof of Theorem~\ref{mainth} by Markus and
Neumann goes as follows. First of all, it is shown that each
canonical region for a flow $(M, \Phi)$ is parallel. (The same
reasoning still works, word by word, for standard regions;
alternatively, notice that each invariant region in a parallel
region is parallel as well.) Here observe that, by a simple
connectedness argument,  all orbits in a canonical (or a standard)
region $\Omega$  share their $\alpha$-limit sets and their
$\omega$-limit sets. Thus it make sense to write $\alpha(\Omega)$
and $\omega(\Omega)$, respectively, to denote them.

Next, under the hypotheses of Theorem~\ref{mainth} for
$(M,\Phi_1)$ and $(M,\Phi_2)$, an easy simplification allows to
assume that both separatrix configurations are equal,
$\mathcal{S}^+:=\mathcal{S}_1^+=\mathcal{S}_2^+$, hence the
canonical regions of $(M,\Phi_1)$ and $(M,\Phi_2)$ are also equal
and the topological equivalence $h:M\to M$ we are looking for
should map each canonical region into itself. Note that the
existence of a toral canonical region implies that
$M=\mathbb{T}^2$; this trivial case can be discarded, for then
both $(\mathbb{T}^2,\Phi_1)$ and $(\mathbb{T}^2,\Phi_2)$ are
equivalent to the rational flow $(\mathbb{T}^2,\Phi_{ss})$.

Now the difficult part of the proof comes (Section~3 in \cite{Ne}
and Section~7 in \cite{Ma}): starting from assuming that $h$ is
the identity on $\mathcal{S}^+$, it must be homeomorphically
extended to each canonical region $\Omega$ (mapping orbits from
$(M,\Phi_1)$ into orbits from $(M,\Phi_2)$  and preserving the
directions). After explaining how this extension must be done, the
authors first check the continuity from ``inside'' at the
so-called accesible regular points from $\Bd \Omega$ (by
\emph{accesible} we mean that there is a lateral tubular region at
the point which is included in $\Omega$), then deduce the
continuity from ``outside'' and at the rest of regular points in
$\mathcal{S}^+$, and finally prove the continuity at the
(isolated) singular points. If fact, the argument equally works
under the weaker hypothesis that the sets of essential singular
points are discrete. Continuity at vertical singular points is
guaranteed from the very beginning, because they are interior to
$\mathcal{S}^+$; on the other hand, maximal curves of horizontal
singular points can be dealt with exactly as if they were regular
orbits.

Unfortunately, in their construction Markus and Neumann take for
granted the following intuitively obvious (but, as shown by the
counterexamples from the previous section, not necessarily true)
fact: if a transversal to a canonical region ends at an accesible
point from its boundary, then the transversal must be
semi-complete. Using standard regions allows to override this
difficulty:

\begin{proposition}\label{keystone}
 Let $\Omega$ be a strip, annular or
radial standard region. If $p\in \Bd \Omega$ is a regular or a
horizontal singular point, and $L\subset \Omega$ is a transversal
ending at $p$ (that is, there is an arc $A$ with endpoint $p$ such
that $A'=A\setminus \{p\}\subset L$), then $L$ is semi-complete;
more precisely, there is a complete transversal to $\Omega$
including $A'$.
\end{proposition}

\begin{proof}
When $\Omega$ is annular, the result is clear.

Assume now that $\Omega$ is a strip and fix a topological
equivalence $h$ between $(\Omega,\Phi)$ and $(\R^2,\Phi_s)$,  when
there is no loss of generality in assuming that $h$ preserves
directions (that is, if $T=h^{-1}(\{0\}\times \mathbb{R})$, then
$\Omega_T^+=h^{-1}((0,\infty)\times \mathbb{R})$) and
$h^{-1}(0,0)=q$ is the other endpoint of the arc $A$.

Let $I=(c,d)$ ($-\infty\leq c<0<d\leq \infty$) be the open
interval and $\mu:I\to \mathbb{R}$ be the continuous map such that
$h(L)=\{(\mu(s),s):s\in I\}$ when, say,  $\lim_{s\to d}
h^{-1}((\mu(s),s))=p$. We argue to a contradiction by assuming
that $d<\infty$.

We claim that either $\lim_{s\to d} \mu(s)=\infty$  or $\lim_{s\to
d} \mu(s)=-\infty$. Otherwise, there would be a sequence $s_n\to
d$ with $\mu(s_n)\to r\in \R$, hence $h^{-1}((s_n,\mu(s_n))$ would
converge both to $h^{-1}(d,r)$, a point in $\Omega$, and to $p$,
which belongs to $\Bd \Omega$. This is impossible.

We suppose, for instance, $\lim_{s\to d} \mu(s)=\infty$. Moveover,
slightly modifying $h$ near $q$ if necessary, we can assume
$\mu(s)>0$ for all $s\in (0,d)$. Hence $T'=h^{-1}(\{0\}\times
(0,\infty))$ does not intersect $A'$.

Let $v=h^{-1}(0,d)$. The orbit $\varphi(v)$ is fine, so there is a
strong strip $S \subset \Omega$ neighbouring it, with its border
orbits also included in $\Omega$, verifying
Definition~\ref{proper}(ii'). Since the points in  $\Cl(S_{T\cap
S}^+)$ which are not in $\Omega$ belong to
$\omega(\Omega)=\omega(v)$, and $p$ is one of such points because
$h^{-1}(\mu(s),s)$ is included in $S$ if $s$ is close enough to
$d$, we get $p\in \omega(\Omega)$.

Fix now a couple of lateral tubular regions $V$ and $W$ at $p$. We
can assume that $V\subset \Omega$ and, moreover, $A \subset V$.
Let $B$ be the corresponding lateral transversal at $p$ included
in $W$. Since $p\in \omega(\Omega)$, all positive semi-orbits
$\varphi(z,+)$, $z\in T$,  must intersect $B$ (in fact, infinitely
many times). Let $q^*$ be the first point from $\varphi(q,+)$ in
$B$, and denote by $B'\subset B$ the transversal with endpoints
$p$ and $q^*$. Now let $T_0'$ (respectively, $T_1'$) be the set of
points $z\in T'$ such that the first intersection point of
$\varphi(z,+)$ with $A'\cup B'$ belongs to $A'$ (respectively, to
$B'$). Both sets are disjoint and non-empty ($v\in T_1'$, and all
points from $T'\cap V$, in particular those close enough to $q$,
belong to $T_0'$), its union is the whole $T'$, and they are
clearly open in $T'$ because the orbit $\varphi(q)$ does not
intersects $T'$. This contradicts the connectedness of $T'$.

Finally, we assume that $\Omega$ is radial and reason again by way
of contradiction, assuming that $A'$ does not intersect all orbits
of $\Omega$ infinitely many times. It is clear that, without loss
of generality, we can suppose that $A'$ does not meet every single
orbit in $\Omega$; with more detail, there is no restriction in
assuming that there exist some $z\in \Omega$ and some strip
neighbourhood of $\varphi_\Phi(z)$, $S$, such that $S \cap A' =
\emptyset$.

Now consider a new flow $\Phi'$ having exactly the same orbits as
$\Phi$ in $M\setminus \varphi_\Phi(z)$ and having $z$ as a
singular point. Then $\varphi_\Phi(z)$, when seen as a subset of
$(M,\Phi')$, consists of three separators for $\Phi'$: the
singular point $z$ and two regular orbits given by the components
of $\varphi_\Phi(z) \setminus\{z\}$. Moreover,
$\Omega'=\Omega\setminus\varphi_\Phi(z)$ is a strip and, clearly,
a standard region for $\Phi'$. The previous argument implies that
$A'$ is semi-complete for $\Phi'$, which is impossible because it
does not intersect $S$.
\end{proof}

With the help of Proposition~\ref{keystone}, Theorem~\ref{elbueno}
can be proved, without further changes, as explained above.

\section*{Acknowledgements}
This work has been partially supported by Ministerio de
Econom\'{\i}a y Competitividad, Spain, grant MTM2014-52920-P. The
first author is also supported by Fundaci\'on S\'eneca by means of
the program ``Contratos Predoctorales de Formaci\'on del Personal
Investigador'', grant 18910/FPI/13.

\end{document}